\documentclass[leqno,11pt]{article}
\usepackage{microtype}
\usepackage{amsfonts,amssymb,amsmath,amsthm,geometry}

\DeclareMathOperator{\tr}{tr}
\DeclareMathOperator{\rank}{rank}

\newcommand{\ud}{\mathrm{d}}

\newtheorem{defi}{Definition}[section]
\newtheorem{proposition}[defi]{Proposition}
\newtheorem{corollary}[defi]{Corollary}

\newcommand{\Iff}{if\textcompwordmark f}

\begin{document}
\title{\textbf{Characterizations of Ruled Surfaces in $\mathbb{R}^3$ \\and of Hyperquadrics in $\mathbb{R}^{n+1}$ \\via Relative Geometric Invariants}}

\author         {\textbf{Stylianos Stamatakis, Ioannis Kaffas} and \textbf{Ioanna-Iris Papadopoulou}
                \\[5mm] Department of Mathematics, Aristotle University of Thessaloniki\\
                GR-54124 Thessaloniki, Greece\\
                e-mail: stamata@math.auth.gr
                }
\date{}
\maketitle
\begin{abstract}
                We consider hypersurfaces in the real Euclidean space $\mathbb{R}^{n+1}$ ($n\geq2$) which are relatively normalized.
                We give necessary and sufficient conditions a) for a surface of negative Gaussian curvature in $\mathbb{R}^3$ to be ruled, b) for a hypersurface of positive Gaussian curvature in $\mathbb{R}^{n+1}$ to be a hyperquadric and c) for a relative normalization to be constantly proportional to the equiaffine normalization.
                \\[2mm] {\em MSC 2010:} 52A15, 52A20, 53A05, 53A07, 53A15, 53A25, 53A40
                \\[2mm] {\em Keywords:} Ruled surfaces, convex hypersurfaces, ovaloids, hyperquadrics, relative normalizations, equiaffine normalization, Pick-invariant
\end{abstract}

\section{Preliminaries}
In this section we shall fix our notation and state some of the most important notions and formulae concerning the relative Differential
Geometry of hypersurfaces in the real Euclidean space $\mathbb{R}^{n+1}$ ($n\geq2$). Our presentation is mainly based on the text \cite{Manhart} and \cite{Schneider1967}. For a more detailed exposition of the subject the reader might read \cite{Schirokow}.

\vspace{2mm}
In the Euclidean space $\mathbb{R} ^{n+1}$ let $\Phi=(M,\bar{x})$, $M\subset\mathbb{R} ^{n}$, be a $C^{r}$-hypersurface defined by an $n$-dimensional, oriented, connected $C^{r}$-manifold $M$ ($r\geq3$) and by a $C^{r}$-immersion  $\bar{x}:M\rightarrow\mathbb{R}^{n+1}$,   whose Gaussian curvature $K_{I}$ never vanishes on $M$.
A $C^{s}$-mapping $\bar{y}:M\rightarrow\mathbb{R}^{n+1}$  ($ r > s \geq 1$) is called a $C^{s}$-\textit{relative normalization}, if
            \begin{subequations}\label{1}
            \begin{align}
            \rank \left(\{\bar{x}_{/1},\bar{x}_{/2}, \dots, \bar{x}_{/n}, \bar{y}\}\right) &= n + 1,\\
            \rank \left(\{\bar{x}_{/1},\bar{x}_{/2}, \dots, \bar{x}_{/n},\bar{y}_{/i}\}\right) &= n, \quad \forall \,\, i = 1,2,\dots,n,
            \end{align}
            \end{subequations}
for all $\left(u^{1},u^{2}, \dots, u^{n} \right) \in M$, where
            \begin{equation*}
            f_{/i}:=\frac{\partial f}{\partial u^{i}},~f_{/ij}:=\frac{\partial ^{2}f}{\partial u^{i}\partial u^{j}} \quad \text{etc.}
            \end{equation*}
denote partial derivatives of a function (or a vector-valued function) $f$. We will also say that the pair $(\Phi,\bar{y})$ is a \emph{relatively normalized hypersurface} of $\mathbb{R} ^{n+1}$. \\
The \emph{covector} $\bar{X}$ of the tangent vector space is defined by%
            \begin{equation}\label{2}
            \langle \bar{X},\bar{x}_{/i}\rangle=0 \quad \text{and}\quad \langle \bar{X},\bar{y}\rangle = 1
            \quad \left(i=1,2,\dots,n\right),
            \end{equation}
where $\langle$ $,$ $\rangle$ denotes the standard scalar product in $\mathbb{R}^{n+1}$.\\
The quadratic differential form
            \begin{equation*}
            G = G_{ij}\,du^{i}\,du^{j}, \quad \text{where} \quad G_{ij}:=\langle \bar{X},\bar{x}_{/ij}\rangle,
            \end{equation*}
is definite or indefinite, depending on whether the Gaussian curvature $K_{I}$ of $\Phi$ is positive or negative, and is called the \emph{relative metric} of $\Phi$.
From now on we shall use $G_{ij}$ as the fundamental tensor for ``raising and lowering the indices''  in the sense of classical tensor notation.\\
Let $\bar{\xi}:M\rightarrow\mathbb{R}^{n+1}$ be the Euclidean normalization of $\Phi$.
By virtue of \eqref{1} the \emph{support function} of the relative normalization $\bar{y}$, which is defined by
            \begin{equation*}
            q:=\langle \bar{\xi},\bar{y}\rangle:M \rightarrow \mathbb{R} ,\quad q\in C^{s}(M),
            \end{equation*}
never vanishes on $M$. In the sequel we choose $\bar{\xi}$ and $\bar{X}$ to have the same orientation. Then $q$ is positive everywhere on $M$.\\
Because of \eqref{2}, it is
            \begin{equation}\label{3}
            \bar{X} = q^{-1} \bar{\xi}, \quad  G_{ij}=q^{-1}h_{ij},  \quad G^{\left( ij \right)} = q\, h^{\left( ij \right)},
            \end{equation}
where $h_{ij}$ are the components of the second fundamental form $II$ of $\Phi$ and $h^{\left( ij \right)} $ resp. $ G^{\left( ij \right)} $ the inverses of the tensors $h_{ ij }$ and $G_{ ij }$.
Let $\nabla^{G}_{i}$ denote the covariant derivative corresponding to $G$.
By
            \begin{equation*}
            A_{jkl}:=\langle \bar{X},\,\nabla^{G}_{l} \, \nabla^{G}_{k}\,\bar{x}_{/j}\rangle
            \end{equation*}
is the (symmetric) \emph{Darboux tensor} defined. It gives occasion to define the \emph{Tchebychev-vector} $\bar{T}$ of the relative normalization $\bar{y}$ \begin{equation*}
            \bar{T}:=T^{m}\, \bar{x}_{/m},\quad \text{where\quad }T^{m}:=\frac{1}{n}A_{i}^{im},
            \end{equation*}%
and the \emph{Pick-invariant}
            \begin{equation*}
            J:=\frac{1}{n \left( n-1 \right)} \, A_{jkl} \, A^{jkl}.
            \end{equation*}
We mention, that when the second fundamental form $II$ is positive definite, so does $G$ and in this case $J \geq 0$ holds on $M$ (see e.g. \cite[p.~133]{Huck}). \\
Denoting by $H_{I}$ the Euclidean mean curvature of $\Phi$, by $\nabla^{II}$ resp. $\triangle^{II}$ the first resp. the second Beltrami differential operator with respect to the fundamental form $II$ of $\Phi$ and by $S_{II}$ the scalar curvature of $II$, the Pick-invariant is computed by (see \cite{Manhart})
            \begin{equation}\label{4}
            J = \frac {3(n+2)}{4n(n-1)}\,q \, \nabla^{II}\left( \ln q, \ln q - \ln |K_{I}|^{\frac{2}{n+2}}\right) +q\, \frac{1}{n\left(n-1\right)} P,
            \end{equation}
where $P$ is the function \cite[p.~231]{Schneider1972}
            \begin{equation}\label{5}
            P  = n\left(n-1\right)\left(S_{II}-H_{I}\right) + \left(2K_{I}\right)^{-2}\nabla^{II}K_{I}.
            \end{equation}
The \emph{relative shape operator} has the coefficients $B_{i}^{j}$ such that
            \begin{equation*}
            \bar{y}_{/i}=:-B_{i}^{j}\, \bar{x}_{/j}.
            \end{equation*}%
The \emph{mean relative curvature}, which is defined by
            \begin{equation*}
            H:=\frac{1}{n}\tr\left(  B_{i}^{j}\right),
            \end{equation*}
is computed by (see \cite{Manhart})
            \begin{equation}\label{6}
            H = q\, H_{I} + \frac {q}{n} \Big[ \triangle^{II}\left( \ln q\right) + \nabla^{II} \left( \ln q,\ ln \left( q\, |K_{I}|^{\frac{-1}{2}} \right) \right) \Big].
            \end{equation}
The scalar curvature $S$ of the relative metric $G$, which is defined formally and is the curvature of the Riemannian or pseudo-Riemannian manifold ($\Phi,G$), the mean relative curvature $H$ and the Pick-invariant $J$ satisfy the \emph{Theorema Egregium of the relative Differential Geometry}, which states that
            \begin{equation}\label{7}
            H + J - S = \frac{n}{n-1}\,\|\bar{T} \|_{G},
            \end{equation}
where $\| \bar{T} \|_{G} := G_{ij}\,T^{i}\,T^{j}$ is the \emph{relative norm} of the Tchebychev-vector $\bar{T}$.
\section{The Tchebychev-function and some related formulae}
We consider the function
            \begin{equation}\label{8}
            \varphi : = \left( \frac {q}{q_\textsc{aff}} \right)^\frac{n+2}{2n},
            \end{equation}
where
            \begin{equation*}
            q_\textsc{aff}:=|K_{I}|^ \frac {1} {n+2}
            \end{equation*}
is the support function of the \emph{equiaffine normalization} $\bar{y}_\textsc{aff}$ and we call it the \emph{Tchebychev-function} of the relative normalization $\bar{y}$.
It is known, that for the components of the Tchebychev-vector holds \cite[p.~199]{Manhart}
            \begin{equation*}
            T^{i}= G^{(ij)} \left( \ln \varphi \right)_{/j}.
            \end{equation*}
Hence, by (\ref{3}c), we obtain
            \begin{equation*}
            \bar{T}= \nabla^{G}\left( \ln \varphi,\bar{x} \right) = q \,  \nabla^{II}\left( \ln \varphi,\bar{x} \right)
            \end{equation*}
and
            \begin{equation}\label{9}
            \left \Vert \bar{T} \right \Vert _{G} = \nabla^{G}\left( \ln \varphi\right) = q \, \nabla^{II}\left( \ln \varphi\right).
            \end{equation}
We notice that the Tchebychev-vector vanishes identically \Iff{} the Tchebychev-function $ \varphi $ is constant, i.e., by \eqref{8}, \Iff{} $q = c\, q_\textsc{aff}$, $c \in \mathbb{R}^{*}$, which means that the relative normalization $\bar{y}$ and the equiaffine normalization $\bar{y}_\textsc{aff}$ are constantly proportional.\\
From the relation \eqref{4} we obtain the Pick-invariant of the Euclidean normalization ($q=1$)
            \begin{equation*}
            J_\textsc{euk} = \frac{1}{n\left( n-1 \right)}P.
            \end{equation*}
Hence by using \eqref{5} we find
            \begin{equation}\label{10}
            J_\textsc{euk}= S_{II}-H_{I}+\frac{(n+2)^2 }{4n(n-1)  }\,\nabla^{II} \left( \ln q_\textsc{aff}\right).
            \end{equation}
From   \eqref{8} and \eqref{10} we conclude that the relation \eqref{4} can be written as
            \begin{equation}\label{11}
            \frac{J}{q} =\frac {3(n+2)}{4n(n-1)}\,\left[\frac{4n^2}{(n+2)^2}\,\nabla^{II} \left(\ln \varphi\right)-\nabla^{II} \left(\ln q_\textsc{aff} \right) \right] + J_\textsc{euk}.
            \end{equation}
For the equiaffine ($\varphi = 1$) Pick-invariant $J_\textsc{aff}$ we deduce
            \begin{equation}\label{12}
            \frac{J_\textsc{aff}}{q_\textsc{aff}}= \frac {-3(n+2)}{4n(n-1)}\,\nabla^{II}\left(\ln q_\textsc{aff}\right) + J_\textsc{euk}.
            \end{equation}
By subtracting \eqref{12} from \eqref{11} we obtain
            \begin{equation}\label{13}
            \frac{J}{q} -\frac{J_\textsc{aff}}{q_\textsc{aff}}= \frac {3n}{(n-1)(n+2)}\,\nabla^{II}\left(\ln \varphi\right).
            \end{equation}
Similarly, taking account of  \eqref{6} and \eqref{8}, we find
            \begin{align}\label{14}
            \frac{H}{q}-H_{I}&=\frac{2}{n+2}\triangle^{II}\left(\ln \varphi\right)+\frac{4n}{(n+2)^2}\,\nabla^{II}\left(\ln \varphi\right)\\
            &- \frac{n-2}{n+2} \,\nabla^{II}\left(\ln \varphi,\ln q_\textsc{aff} \right)+ \frac{1}{n} \triangle^{II}\left(\ln q_\textsc{aff}\right)-\frac{1}{2}\,\nabla^{II}\left(\ln q_\textsc{aff}\right). \notag
            \end{align}
For the mean equiaffine curvature $H_{\textsc{aff}}$ we infer
           \begin{equation}\label{15}
           \frac{H_\textsc{aff}}{q_\textsc{aff}}-H_{I}=\frac{1}{n} \triangle^{II}\left(\ln q_\textsc{aff}\right)-\frac{1}{2}\,\,\nabla^{II}\left(\ln  q_\textsc{aff}\right).
           \end{equation}
By subtracting \eqref{15} from \eqref{14} we obtain
            \begin{equation}\label{16}
            \frac{H}{q}-\frac{H_\textsc{aff}}{q_\textsc{aff}}=\frac{2}{n+2}\triangle^{II}\left(\ln \varphi\right)+\frac{4n}{(n+2)^2}\,\nabla^{II}\left(\ln  \varphi\right) - \frac{n-2}{n+2} \,\nabla^{II}\left(\ln \varphi,\ln q_\textsc{aff} \right).
            \end{equation}
The relations \eqref{7}, \eqref{9}, \eqref{13} and \eqref{16} may be combined into
            \begin{equation*}
            \frac{S}{q} - \frac{J_\textsc{aff} + H_\textsc{aff}}{q_\textsc{aff}} = \frac{2}{n+2} \triangle^{II}\left( \ln \varphi\right) -\frac{n(n-2)}{(n+2)^2}\nabla^{II}\left( \ln \varphi\right)- \frac{n-2}{n+2} \,\nabla^{II}\left(\ln \varphi,\ln q_\textsc{aff} \right)
            \end{equation*}
and with reference to
            \begin{equation}\label{17}
            S_\textsc{aff} = J_\textsc{aff} + H_\textsc{aff},
            \end{equation}
where  $S_\textsc{aff}$ denotes the inner equiaffine curvature, we conclude that
            \begin{equation}\label{18}
            \frac{S}{q} - \frac{S_\textsc{aff}}{q_\textsc{aff}} = \frac{2}{n+2} \triangle^{II}\left( \ln \varphi\right) -\frac{n(n-2)}{(n+2)^2}\nabla^{II}\left( \ln \varphi\right)- \frac{n-2}{n+2} \,\nabla^{II}\left(\ln \varphi,\ln q_\textsc{aff} \right).
            \end{equation}
\section{Characterizations of ruled surfaces of $\mathbb{R}^{3}$ and of hyperquadrics of $\mathbb{R}^{n+1}$}
Let now $\alpha$ be any real number.
By using the relations \eqref{13} and   \eqref{16}--\eqref{18} we obtain
            \begin{align*}
            \frac{\alpha \left ( S - H \right) + J}{q}
            =  \left( \alpha + 1 \right)\frac {J_{\textsc{aff}}} {q_{\textsc{aff}}} - \frac {n \left[ \alpha \left( n - 1 \right) - 3 \right] }{\left( n - 1 \right) \left( n + 2 \right)} \nabla^{II}\ln \varphi. \notag
            \end{align*}
For $\alpha = \frac {3}{n - 1} $ we get
            \begin{equation}\label{19}
            \frac{3 \left( S - H \right) + \left( n - 1 \right) J }{q} =  \left( n + 2 \right)
            \frac{J_\textsc{aff}} {q_\textsc{aff}}.
            \end{equation}
This result implies the following
            \begin{proposition}\label{P:3}
            Let $(\Phi,\bar{y})$ be a relatively normalized hypersurface of \/ $\mathbb{R}^{n+1}$. Then the function
            \begin{equation*}
            \frac{3 \left( S - H \right) + \left( n - 1 \right) J }{q}
            \end{equation*}
            is independent of the relative normalization and vanishes \Iff{} $J_\textsc{aff}=0$.
            \end{proposition}
On account of the relations \eqref{7} and \eqref{19} we infer that
            \begin{equation}\label{20}
            \|\bar{T} \| _{G} = \frac {\left(n - 1 \right) \left(n + 2 \right)} {3n} \left( J - \frac {q} {q_{\textsc{aff}}} J_{\textsc{aff}} \right) = \frac {n + 2} {n} \left( H - S + \frac {q} {q_{\textsc{aff}}} J_\textsc{aff}\right).
            \end{equation}
From \eqref{20} it follows immediately that
            \begin{equation}\label{21}
            J_{\textsc{aff}} = 0 \iff  3n \left \Vert T \right \Vert _{G} = \left( n - 1 \right) \left( n + 2 \right) J \iff
            n \left \Vert T \right \Vert _{G} = \left( n + 2 \right) \left( H - S \right) .
            \end{equation}
We suppose that $n = 2$ and $K_{I} < 0$.
It is well known (see~\cite[p.~125]{Blaschke}), that the vanishing of $J_\textsc{aff}$ characterizes the ruled surfaces of $\mathbb{R}^{3}$ among the surfaces of negative Gaussian curvature. So,  from the relations \eqref{19} and \eqref{21} we obtain the following characterizations for  ruled surfaces in $\mathbb{R}^{3}$:
            \begin{proposition}\label{P:4}
            Let $\Phi \subset \mathbb{R}^{3}$ be a surface of negative Gaussian curvature.
            Then \/ $\Phi$ is a ruled surface \Iff{} there exists a relative normalization of \/ $\Phi$, for which one of the following equivalent properties holds true:
            \begin{description}
            \item[$(a)$]  \qquad    $   3 \left( S - H \right) +J$=0,
            \item[$(b)$]  \qquad    $   3\|\bar{T}\|_{G} = 2 J $,
            \item[$(c)$]  \qquad    $   \|\bar{T}\|_{G} = 2 \left( H - S \right)$.
            \end{description}
            \end{proposition}

\vspace{2mm}
Let now be $n \geq 2$ and $K_{I} > 0$.
Moreover, without loss of generality, we assume that the second fundamental form $II$ is positive definite.
It is also well-known (see \cite[p.~380]{Schneider1967}) that in this case the equiaffine Pick-invariant is non-negative and that it vanishes \Iff{} $\Phi$ is a hyperquadric.
So,  by using the relations \eqref{19} and \eqref{21}, we can characterize the hyperquadrics of $\mathbb{R}^{n+1}$ among all hypersurfaces of positive Gaussian curvature as the following proposition states:
           \begin{proposition}\label{Prop3}
           Let $\Phi \subset \mathbb{R}^{n+1}$ be a hypersurface of positive Gaussian curvature.
           Then $\Phi$ is a hyperquadric \Iff{} there exists a relative normalization of  $\Phi$, for which one of the following equivalent properties holds true:
           \begin{description}
           \item[$(a)$]  \qquad     $3 \left( S -H \right) + \left( n-1 \right) J = 0$,
           \item[$(b)$]  \qquad     $3n \left \Vert T \right \Vert _{G} = \left( n - 1 \right) \left( n + 2 \right) J$ ,
           \item[$(c)$]  \qquad     $ n \left \Vert T \right \Vert _{G} = \left( n + 2 \right) \left( H - S \right) $.
           \end{description}
           \end{proposition}
\section{The vanishing of the Pick-invariant and some integral formulae}
Another consequence of relation \eqref{13} are the following two propositions:
            \begin{proposition}
            Let  $\Phi \subset  \mathbb{R}^{n+1}$ be a hypersurface of positive Gaussian curvature.
            For the Pick-invariant of every relative normalization $\bar{y}$ the following relation is valid
            \begin{equation}\label{22}
            \frac{J}{q} -\frac{J_\textsc{aff}}{q_\textsc{aff}} \geq 0.
            \end{equation}
            The equality holds \Iff{} the relative normalization $\bar{y}$ and the equiaffine normalization $\bar{y}_\textsc{aff}$ are constantly proportional.
            \end{proposition}
            \begin{proof}
            Because of the assumption $K_{I} > 0$, it is $\nabla^{II}\left(\ln \varphi \right)\geq 0$, so from \eqref{13} it follows the inequality.
            Furthermore it is
            \begin{equation*}
            \frac{J}{q} -\frac{J_\textsc{aff}}{q_\textsc{aff}} = 0 \Leftrightarrow \nabla^{II}\left(\ln \varphi\right) = 0 \Leftrightarrow \varphi = const. \Leftrightarrow q = c \,{q_\textsc{aff}}, c \in~\mathbb{R}^{\ast},
            \end{equation*}
            which proves the assertion.
            \end{proof}
            \begin{proposition}
            Let  $\Phi \subset  \mathbb{R}^{n+1}$ be a hypersurface of positive Gaussian curvature.
            If there is a relative normalization $\bar{y}$, whose Pick-invariant vanishes identically, then $\Phi$ is a hyperquadric. Furthermore $\bar{y}$ is constantly proportional to the equiaffine normalization $\bar{y}_\textsc{aff}$.
            \end{proposition}
            \begin{proof}
            Let $\bar{y}$ be a relative normalization of $\Phi$ with vanishing Pick-invariant. Then, from the relation \eqref{13} we obtain
            \begin{equation}\label{23}
            -\frac{J_\textsc{aff}}{q_\textsc{aff}}= \frac {3n}{(n-1)(n+2)}\,\nabla^{II}\left(\ln \varphi\right).
            \end{equation}
            Because of  $J_\textsc{aff} \geq 0$ and $\nabla^{II}\ln \varphi \geq 0$, both members of \eqref{23} vanish.
            But $J_\textsc{aff} \geq 0$ implies that $\Phi$ is a hyperquadric and  $\nabla^{II}\ln \varphi = 0$ implies that the function $\varphi $ is constant, which means that $q = c \,{q_\textsc{aff}}, c \in~\mathbb{R}^{\ast}$ and the proof is completed.
            \end{proof}
We conclude the paper by considering closed surfaces of positive Gaussian curvature (ovaloids) in $\mathbb{R}^{3}$.
For $n = 2$ relation  \eqref{16} becomes
            \begin{equation*}
            \frac{H}{q}-\frac{H_\textsc{aff}}{q_\textsc{aff}}=\frac{1}{2}\triangle^{II}\left(\ln \varphi\right)+\frac{1}{2}\,\nabla^{II}\left(\ln  \varphi\right),
            \end{equation*}
from which we have
            \begin{proposition}
            Let $(\Phi,\bar{y})$ be a relatively normalized ovaloid in $\mathbb{R}^{3}$. Then
            \begin{equation*}
            \iint _{M} \left(\frac{H}{q} - \frac{H_\textsc{aff}}{q_\textsc{aff}}\right) \, \ud O_{II} \geq 0,
            \end{equation*}
            where $ \ud O_{II}$ is the element of area of $\Phi$ with respect to the second fundamental form $II$ of $\Phi$. The equality is valid \Iff{} the relative normalization $\bar{y}$ is constantly proportional to the equiaffine normalization $\bar{y}_\textsc{aff}$.
            \end{proposition}
Furthermore, for $n=2$, relation \eqref{18} becomes
            \begin{equation}\label{24}
            \frac{S}{q}-\frac{S_\textsc{aff}}{q_\textsc{aff}}=\frac{1}{2} \triangle^{II}\left( \ln \varphi\right).
            \end{equation}
From this equation we easily deduce:
            \begin{proposition}
            Let $(\Phi,\bar{y})$ be a relatively normalized ovaloid of $\mathbb{R}^{3}$. If the function
            \begin{equation*}
            \frac{S}{q}-\frac{S_\textsc{aff}}{q_\textsc{aff}}
            \end{equation*}
            does not change its sign on M, then the relative normalization $\bar{y}$ and the equiaffine normalization $\bar{y}_\textsc{aff}$ are constantly proportional.
            \end{proposition}
Finally, from the relations \eqref{10}, \eqref{12}, \eqref{15} and \eqref{17} for $n=2$ we obtain
            \begin{equation}\label{25}
            \frac{S_\textsc{aff}}{q_\textsc{aff}}-S_{II}=\frac{1}{2} \triangle^{II}\left( \ln q_\textsc{aff} \right).
            \end{equation}
If we now integrate \eqref{24} and \eqref{25} over $M$ we get
            \begin{equation*}
            \iint _{M} \frac{S}{q} \, \,  \ud O_{II} = \iint _{M} \frac{S_\textsc{aff}}{q_\textsc{aff}} \, \, \ud O_{II} = \iint _{M} S_{II} \, \, \ud O_{II} = 2\pi\chi,
            \end{equation*}
where $\chi$ is the Euler characteristic of $\Phi$.
Hence we arrive at
            \begin{proposition}
            Let $\Phi$ be a relatively normalized ovaloid of $\mathbb{R}^{3}$. Then the following integral formula is valid
            \begin{equation*}
            \iint _{M} \frac{S}{q} \, \,  \ud O_{II} = 2\pi\chi,
            \end{equation*}
            where $\chi$ is the Euler characteristic of $\Phi$.
            \end{proposition}
            \begin{corollary}
            For an ovaloid $\Phi \subset\mathbb{R}^{3}$ the following integral formula is valid
            \begin{equation*}
            \iint _{M} \frac{S_\textsc{aff}}{q_\textsc{aff}} \, \, \ud O_{II} = 2\pi\chi.
            \end{equation*}
            \end{corollary}

\end{document}